\documentclass{birkjour}
\usepackage[T1]{fontenc}
\usepackage{cite}
\usepackage{amscd,amssymb,amsmath,color}
\usepackage[all]{xy}

\def\O{\mathcal O}

\def\AA{{\mathbb A}}

\def\CC{{\mathbb C}}

\def\FF{{\mathbb F}}

\def\tt{{\mathfrak T}}

\newcommand{\Cc}{\mathcal{C}}

\def\O{{\bf O}}

\def\*C{{}^*\hspace*{-1pt}{\Cc}}

\def\text#1{{\rm {\rm #1}}}

 \def\1{\mathbf{1}}

\def\sna#1#2{\mathrm{sna}{(#1 , #2)}}
\def\sl#1#2{\mathrm{sl}{(#1 , #2)}}

\newtheorem{theorem}{Theorem}[section]

\newtheorem{proposition}[theorem]{Proposition}
\theoremstyle{definition}
\newtheorem{definition}[theorem]{Definition}
\theoremstyle{remark}
\newtheorem{remark}[theorem]{Remark}

\newtheorem*{example}{Example}
\numberwithin{equation}{section}

\newcounter{rlist}

\begin{document}

\title{Special normalised affine matrices. An example of a  Lie affgebra.}

 \author{Tomasz Brzezi\'nski}
 \address{Department of Mathematics, Swansea University\\
Fabian Way,
Swansea
SA1 8EN, U.K.  \\
\and\\ Department of Mathematics, University of Bia\l ystok\\ K.\ Cio\l kowskiego 1M, 15--245 Bia\l ystok, Poland} 
  \email{T.Brzezinski@swansea.ac.uk}   

\thanks{The research  is partially supported by the National Science Centre, Poland, grant no.\ 2019/35/B/ST1/01115.}
\dedicatory{In memory of Anatol Odzijewicz}


\begin{abstract}
The affine space of traceless complex matrices in which the sum of all elements in every row and every column is equal to one is presented as an example of an affine space with a Lie bracket or a Lie affgebra. \end{abstract}

\maketitle

\section{Introduction}
A choice of frame of reference  to describe classical motion of particles in space amounts to choosing a particular co-ordinate system or a vector space. The origin of the co-ordinate system (or the zero vector in the vector space) is determined by a position of the observer. The Tulczyjew programme of a frame independent approach to classical mechanics \cite{Tul:fra} postulated to replace vector spaces (co-ordinate systems) by affine spaces; this is most natural since there are no distinguished points in an affine space. The implementation of this programme involved in particular developing the theory of bundles with affine fibres (rather than vector bundles) \cite{GraGra:fra} and more generally the AV-differential geometry \cite{GraGra:av1}, \cite{GraGra:av2}. Along the way a notion of a Lie bracket on an affine space or a {\em Lie affgebra} was proposed in \cite{GraGra:Lie}.  As impressive as these developments are they still treat an affine space relative to a distinguished vector space and make an extensive use of this space. One might argue that at this point a choice of a frame reference is made and the departure from the spirit of the Tulczyjew programme becomes inevitable.

In recent years a deeper understanding of algebraic structures on affine spaces defined with no reference to a vector spaces was achieved (see e.g.\ \cite{Brz:tru}, \cite{Brz:par}, \cite{BreBrz:hea}). Guided by this a proposal for a Lie algebra structure on an intrinsically defined affine space was made in \cite{BrzPap:Lie}. The aim of this note is to summarise both the vector-space free approach to affine spaces and the proposal for a Lie bracket with no reference to a vector space made in \cite{BrzPap:Lie}, and to illustrate it by an explicit example of a Lie structure on an affine space of special normalised affine matrices, that is, traceless complex matrices with sums of elements in each column and in each row being equal to one.

\section{Affine spaces without vector spaces}
Traditionally an affine space over a field $\FF$ is defined as a pair $(A,\overset\rightarrow A)$ in which $A$ is a set and $\overset\rightarrow A$ is a vector space which acts freely and transitively on the set $A$. Denoting this action by $+$, for any vector $v\in \overset\rightarrow A$ and any $a\in A$, there is a unique point $a+v\in A$, and for any $a,b\in A$ there is a unique vector $\overset\longrightarrow{ab}\in \overset\rightarrow A$ such that $b=a+\overset\longrightarrow{ab}$. An affine transformation from $(A,\overset\rightarrow A)$   to $(B,\overset\rightarrow B)$ is a function $f:A\to B$ for which there exists a necessarily unique linear transformation $\hat f : \overset\rightarrow A \to \overset\rightarrow B$ such that
\begin{equation}\label{affine.tran.v}
\hat f \big(\overset\longrightarrow{ab}\big) = \overset{-\!-\!-\!-\!\longrightarrow}{f(a)f(b)}, \qquad \mbox{for all $a,b\in A$}.
\end{equation}

It is a well known, but for quite a while overlooked, fact (see \cite{BreBrz:hea} for a discussion and references) that one can define affine spaces intrinsically, with no reference to vectors. We explain this presently. 

The first observation is that although there is no natural binary operation on an affine space, the combination of any three elements $a,b,c\in A$:
\begin{equation}\label{affine.heap}
\langle a,b,c\rangle := a + \overset\longrightarrow{bc},
\end{equation}
makes a well-defined function $A\times A\times A\to A$, that equips $A$ with the structure of  an {\em abelian heap}, a notion introduced by Pr\"ufer in \cite{Pru:the}. That is, for all $a,b,c,d,e\in A$,
\begin{equation}\label{heap}
\langle\langle a,b,c\rangle,d,e\rangle = \langle a,b,\langle c,d,e\rangle\rangle, \quad \langle a,a,b\rangle = b, \quad \langle a,b,c\rangle = \langle c,b,a\rangle.
\end{equation}
In view of the first and the last of conditions \eqref{heap} the distribution of the brackets $\langle,\rangle$ in-between an odd number of elements of $A$ does not matter, hence there is no need to write any internal brackets indicating multiple application of the heap operation.

Any abelian group $A$ is a heap with operation $\langle a,b,c\rangle = a-b+c$. Conversely, given any (non-empty) heap $A$ and $o\in A$, the operation 
\begin{equation}\label{plus}
a+b = \langle a,o,b\rangle
\end{equation}
 makes $A$ into an abelian group which we will denote by $A_o$. One easily checks that $o$ is a neutral element in $A_o$ and that $-a = \langle o,a,o\rangle$. Furthermore, 
 \begin{equation}\label{+-}
 \langle a,b,c\rangle = a-b+c \quad \mbox{(in $A_o$)}. 
 \end{equation}
 Because of this correspondence, the reader might find it more convenient to view abelian heaps as abelian groups with a single ternary operation defined by a combination of addition and subtraction as in \eqref{+-}.
 
 A {\em homomorphism of heaps} is a mapping $f: A\to B$ preserving ternary operations, that is 
\begin{equation}\label{heap.mor}
f(\langle a,b,c\rangle) = \langle f(a),f(b),f(c)\rangle, \quad \mbox{for all $a,b,c\in A$}.
\end{equation}
A homomorphism of abelian groups is automatically a homomorphism of corresponding heaps, but not vice versa.
More information about heaps can be found for example in \cite[Section~2]{Brz:par}.

Returning to an affine space $(A,\overset\rightarrow A)$, the combination of a scalar $\lambda \in \FF$ with any two elements $a,b\in A$, 
\begin{equation}\label{action}
\lambda\triangleright_a b:= a+ \lambda \overset\longrightarrow{ab},
\end{equation}
defines a function $\FF\times A\times A\to A$, $(\lambda,a,b)\mapsto \lambda\triangleright_a b$, satisfying the following conditions, for all $\lambda,\mu,\nu \in \FF$ and $a,b,c,d\in A$:
\begin{subequations}\label{action.ax}
\begin{equation}\label{heap.ax}
(\lambda - \mu + \nu)\triangleright_a b = \langle \lambda \triangleright_a b,  \mu \triangleright_a b, \nu \triangleright_a b\rangle, \;\;\; \lambda \triangleright_a \langle b,c,d\rangle = \langle \lambda \triangleright_a b,  \lambda \triangleright_a c, \lambda \triangleright_a d\rangle,
\end{equation}
\begin{equation}\label{ass.un}
(\lambda \mu)\triangleright_a b = \lambda \triangleright_a (\mu\triangleright_a b), \quad 1\triangleright_a b = b, \quad 0 \triangleright_a b =a,
\end{equation}
\begin{equation}\label{base.change}
\lambda \triangleright_a b = \langle \lambda \triangleright_c b, \lambda \triangleright_c a, a\rangle.
\end{equation}
\end{subequations}

Equivalently to the traditional definition and fully intrinsically, an {\em affine space} over $\FF$ can be defined as a set $A$ together with operations 
\begin{equation}\label{affine.oper}
\langle -,-,-\rangle: A^3\to A \quad \mbox{and}\quad  -\triangleright_{-} - :\FF\times A^2 \to A,
\end{equation}
 which satisfy conditions \eqref{heap} and \eqref{action.ax}. Given affine spaces $A$ and $B$ an affine map from $A$ to $B$ is a function $f:A\to B$ that in addition to preserving the heap operations as in \eqref{heap.mor} preserves the actions, that is, for all $a,b\in A$ and $\lambda\in \FF$,
\begin{equation}\label{mor.act}
f(\lambda \triangleright_a b) = \lambda \triangleright_{f(a)}f(b) .
\end{equation}
Given such defined affine space $A$ a vector space can be assigned to \textbf{any} choice of an element of $o\in A$. The addition of vectors is defined by the formula \eqref{plus}, while multiplication by scalars  is defined by
\begin{equation}\label{scalar}
\lambda a = \lambda \triangleright_o a.
\end{equation}
We denote this vector space by $\mathcal{V}(A_o)$. An affine map $f: A\to B$ induces a unique linear transformation  $\hat f: \mathcal{V}(A_o) \to \mathcal{V}(B_{\tilde{o}})$ by the formula:
\begin{equation}\label{aff.lin}
\hat f(a) = f(a) - f(o),
\end{equation}
in accord with the traditional definition of an affine map \eqref{affine.tran.v}.

This intrinsic approach to affine spaces is in perfect concord with the principle mentioned in Introduction, according to which classical physical phenomena occur, and hence can be described, independently of an observer. Introduction of an observer into the system amounts to the specification of a point of reference or the zero of the vector space. Vector spaces corresponding to different observers (choices of the zero vector $o$), albeit not necessarily identical, are mutually isomorphic, and hence descriptions by different observers of the same physical phenomena are mutually equivalent.

The  {\em dimension} of an affine space $A$ is defined as the dimension of any of the  vector spaces $\mathcal{V}(A_o)$.

\section{Lie affgebras and their reduction to Lie algebras}
The definition of a Lie bracket on an affine space traditionally defined as a pair $(A,\overset \to A)$ proposed in \cite{GraGra:Lie} makes heavy use of the linear structure of $\overset \to A$.  A {\em vector-valued Lie bracket} on $(A,\overset \to A)$ is an anti-symmetric bi-affine map
    $$
    [-,-]_{v}: A\times A \longrightarrow \overset{\rightarrow}{A},
    $$
    that satisfies the Jacobi identity in $\overset{\rightarrow}{A}$:
    \begin{equation}\label{Jacobi.v}
        \widehat{[a,}[b,c]_v]_v + \widehat{[b,}[c,a]_v]_v + \widehat{[c,}[a,b]_v]_v = 0,
    \end{equation}
    where $\widehat{[a,}-]_v$ is the linearisation of the map  $[a,-]_v$ etc.

The aim of the intrinsic definition of a Lie affgebra in \cite{BrzPap:Lie} is to form a Lie structure on an affine space with no reference to a vector space. Such a structure should include the vector-valued Lie brackets of \cite{GraGra:Lie} as a possibly special case. In addition, in the same way as a heap can be retracted to a group by making a choice of an element, it is desired the Lie bracket on an affine space to reduce to a usual Lie algebra structure for any choice of an underlying vector space.
\begin{definition}[\cite{BrzPap:Lie}]\label{def.Lie}
Let $A$ be an affine space over $\FF$. A {\em Lie bracket} on $A$ is a bi-affine operation $[-,-]: A\times A \to A$ such that, for all $a,b\in A$,
\begin{subequations}\label{Lie}
\begin{equation}\label{anti}
\big\langle [a,b], [a,a], [b,a] \big\rangle = [b,b],
\end{equation}
\begin{equation}\label{Jacobi}
\big\langle [a,[b,c]], [a,a], [b,[c,a]], [b,b], [c,[a,b]] \big\rangle = [c,c].
\end{equation}
\end{subequations}
An affine space with a specified Lie bracket is called a {\em Lie affgebra}. Given two Lie affgebras, an affine map that respects Lie brackets is called a {\em Lie affgebra map}.
\end{definition}

It might be worth pointing out that the condition \eqref{anti} that replaces the antisymmetry is invariant under the exchange $a\leftrightarrow b$, while the Jacobi \eqref{Jacobi} is invariant under the cyclic operation:
$$
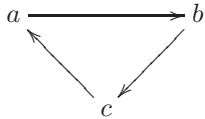
\xymatrix{ a \ar[rr] && b \ar[dl]\\ & c\ar[ul]  & }
$$

By making the choice of $o\in A$ and setting $\overset \to A = A_o$ one recovers the vector-valued Lie bracket  on $(A,\overset \to A)$ through the correspondence
\begin{equation}\label{Lie.vv}
[a,b]_v = [a,b] - b,
\end{equation} 
provided that $[a,a]=a$, for all $a\in A$.

\begin{proposition}[\cite{BrzPap:Lie}]\label{prop.Lie}
Let $A$ be an affine space and $o\in A$. Given Lie bracket $[-,-]$ on $A$ define
\begin{equation}\label{Lie.o}
[a,b]_o = \big\langle [a,b], [a,o],[o,o],[o,b], o\big\rangle, \qquad \mbox{for all $a,b\in A$}.
\end{equation}
Then the $\mathcal{V}(A_o)$ is a Lie algebra with the bracket $[-,-]_o$. We denote it by $\mathcal{L}(A;o)$.
\end{proposition}

Rather than giving a proof (which can be found in \cite{BrzPap:Lie}) we will just make a few comments. First, since the formula \eqref{Lie.o} refers to the abelian group $A_o$ (with $o$ the neutral element), its presentation in $\mathcal{V}(A_o)$ takes the simple form
\begin{equation}\label{Lie.ab}
[a,b]_o = [a,b] - [a,o] + [o,o] -[o,b].
\end{equation}
The definition of $[a,b]_o$ is obtained by applying the linearising equation \eqref{aff.lin} to both arguments of the bi-affine map $[-,-]$. Hence $[-,-]_o$ is a bilinear operation on $A_o$. The antisymmetry follows from \eqref{anti} and the Jacobi identity from \eqref{Jacobi}.

\begin{remark}\label{rem.line}
The procedure described in Proposition~\ref{prop.Lie} might reduce non-isomorphic Lie affgebras  to isomorphic Lie algebras. Consider an affine line $\mathbb{A}$, i.e.\ an affine space generated by two elements $a\neq b$,
$
\mathbb{A} = \{ \lambda\triangleright_a b \;|\; \lambda \in \FF\}.
$  Note that $\lambda\triangleright_a b = \mu\triangleright_a b$ if and only if $\lambda = \mu$.
For any $o\in \AA$, $\mathbb{A}$ reduces to a one-dimensional vector space an therefore Proposition~\ref{prop.Lie} yields only one Lie algebra structure on $\mathcal{V}(\AA_o)$ (the zero Lie bracket). On the other hand, for any $\zeta\in \FF$, 
\begin{equation}\label{line}
[a,a] = a, \qquad [a,b] = \zeta\triangleright_a b, \qquad [b,b] = b,
\end{equation}
defines a Lie bracket on the affine space $\AA$\footnote{The definition \eqref{line} is equivalent to say that $[x,y]= \zeta\triangleright_x y$, for all $x,y\in \AA$. The latter Lie bracket can be defined on any affine space, not necessarily one-dimensional one.}. Let $[-,-]_i$ denote the bracket corresponding to $\zeta_i$, $i=1,2$. Any affine map $f:\AA\to \AA$ is fully determined by its values on $a$ and $b$. Suppose that
\begin{equation}\label{aff.iso}
f(a) = \lambda\triangleright_a b, \qquad f(b) = \mu\triangleright_ab, \qquad \lambda,\mu \in \FF.
\end{equation}
Quite lengthy but straightforward algebraic manipulations (that use the definitions of an affine space and a Lie bracket) allow one to conclude that
$[f(a),f(b)]_2 = f([a,b]_1)$ if and only if 
\begin{equation}\label{equal}
(\mu\zeta_2 - \lambda\zeta_2 +\lambda) \triangleright_a b = (\mu\zeta_1 - \lambda\zeta_1 +\lambda) \triangleright_a b,
\end{equation}
so that $(\zeta_1-\zeta_2)\mu = (\zeta_1-\zeta_2)\lambda$. Hence either $\zeta_1 =\zeta_2$ or $\lambda =\mu$. The latter means that $f$ is not onto and hence not an isomorphism either. Therefore, different choices of $\zeta$ yield non-isomorphic Lie affgebras on $\AA$ with the brackets defined by \eqref{line}.
\end{remark}

\section{The special normalised complex affine Lie affgebras.}
In this section we discuss a family of explicit examples of Lie affgebras. 
\begin{proposition}\label{prop.san}
Let
\begin{equation}\label{san}
\sna n \CC := \{a\in \sl{n\!+\!1} \CC \;|\; \forall j, \; \sum_{i=1}^na_{ij} = 1=\sum_{i=1}^na_{ji} \}.
\end{equation}
Then $\sna n \CC$ is an $n^2-1$-dimensional Lie affgebra with the affine space structure
\begin{equation}\label{san.aff}
\langle a, b ,c\rangle = a-b+c, \qquad \lambda\triangleright_ab = \lambda b + (1-\lambda)a,
\end{equation}
and the Lie bracket
\begin{equation}\label{san.Lie}
[a,b] = ab-ba+b,
\end{equation}
where addition, multiplication (by scalars) etc.\ are the standard operations on matrices.
\end{proposition}
\begin{proof}
Although neither addition nor subtraction of matrices in $\sna n \CC$ gives a matrix in $\sna n \CC$, the combination of addition and subtraction \eqref{san.aff} does, and hence $\sna n \CC$ is an abelian heap. Similarly, $\lambda b + (1-\lambda)a$ is a traceless matrix in which the sum of all elements in each row and all elements in each column equals one, and so is an element of $\sna n \CC$. Clearly, for all $a,b,c,d\in \sna n \CC$  and $\lambda, \mu\in \CC$,
$$
\begin{aligned}
\lambda\triangleright_a(b-c+d) &=\lambda\triangleright_ab-\lambda\triangleright_ac +\lambda\triangleright_ad,
\end{aligned}
$$
and 
$$
\begin{aligned}
\mu \triangleright_{\lambda \triangleright_ab}(\lambda \triangleright_ac) &= \lambda \mu c + \mu(1-\lambda)a +(1-\mu)\lambda b +(1-\mu)(1-\lambda)a \\
&= \lambda \mu c + (1-\mu)\lambda b +(1-\lambda)a  = \lambda\triangleright_a(\mu \triangleright_bc),
\end{aligned}
$$
and hence $\lambda\triangleright_a-:A\to A$ is an affine map. Similarly one checks that $-\triangleright_ab:\CC\to A$ is an affine map. The conditions \eqref{ass.un} and \eqref{base.change} are verified by a routine calculation. 

Any $(n+1)\times (n+1)$ matrix with freely chosen entries denoted by asterisks
$$
\begin{pmatrix}
* & * & * & \ldots & * & \square\cr
* & * & * & \ldots & * & \square\cr
\ldots & \ldots &\ldots & \ldots & \ldots & \ldots \cr
* & * & * & \ldots & * & \square\cr
\square & * & * & \ldots & * & \square\cr
\square & \square & \square& \ldots & \square & \square
\end{pmatrix},
$$
can be uniquely completed to a matrix in $\sna n \CC$; hence  
 $\dim  \sna n \CC = n^2 -1$. Next, since $[a,a]=a$ we find
$$
[a,b]-[a,a]+[b,a] = ab-ba+b -a +ba -ab +a = b = [b,b].
$$
Finally, the Jacobi identity \eqref{Jacobi} follows from the fact that the commutator satisfies the standard Jacobi identity and the observation that $[a,b]-[b,b]=ab-ba$.
\end{proof}

\begin{example}\label{ex.san}
There is only one $2\times 2$ traceless matrix with sums of elements in each row and each column equal to one, and hence
$$
\sna 1 \CC = \left\{\begin{pmatrix} 0 &1 \cr 1 &0\end{pmatrix}\right\}. 
$$

For $n=2$, every element of $\sna 2 \CC$ has the form
$$
A^{ab}_{~c}:= \begin{pmatrix} a & b & 1-a-b\cr 1- 2a-b-2c & c & 2a+b+c\cr a+b+2c & 1-b-c &-a-c\end{pmatrix},
$$
where $a,b,c\in \CC$. Therefore $\sna 2 \CC$  is a barycentric (i.e.\ given as a linear combination with scalars adding up to 1) span of 
$$
A^{00}_{~0} =\begin{pmatrix} 0 & 0 & 1\cr 1 & 0 & 0\cr 0 & 1& 0\end{pmatrix},\qquad 
A^{01}_{~0} = \begin{pmatrix} 0 & 1 & 0\cr 0 & 0 & 1\cr 1 & 0& 0\end{pmatrix},
$$
$$
A^{00}_{~1}=\begin{pmatrix} 0 & 0 & 1\cr -1 & 1 & 1\cr 2 & 0& -1\end{pmatrix},\qquad 
A^{10}_{~0}=\begin{pmatrix} 1 & 0 & 0\cr -1 & 0 & 2\cr 1 & 1& -1\end{pmatrix}.
$$
The action of the Lie bracket on these generators comes out as:
\begin{subequations}\label{san2}
\begin{equation}
 [A^{01}_{~0},A^{00}_{~1}]=-2A^{10}_{~0}+A^{01}_{~0}+2A^{00}_{~1},\quad [A^{00}_{~0},A^{00}_{~1}]=2A^{10}_{~0}-A^{01}_{~0}, \quad
\end{equation}
\begin{equation}
 [A^{00}_{~0},A^{10}_{~0}] = 2A^{10}_{~0}+A^{01}_{~0}-2A^{00}_{~1},\quad [A^{01}_{~0},A^{10}_{~0}] = 2A^{00}_{~1}-A^{01}_{~0},
\end{equation}
\begin{equation}
[A^{00}_{~0},A^{01}_{~0}]=A^{01}_{~0},\; \;[A^{00}_{~1},A^{10}_{~0}]=2A^{10}_{~0}+A^{01}_{~0}+A^{00}_{~1}-3A^{00}_{~0}.
\end{equation}
\end{subequations}
\end{example}

\begin{proposition}\label{prop.san.lin}
For any $o\in \sna n \CC$, the Lie algebra $\mathcal{L}(\sna n \CC ;o)$ is isomorphic to the Lie sub-algebra $\sl{n+1}\CC_0$ of $\sl{n+1}\CC$ consisting of all matrices with vanishing sums of elements in every row and every column:
\begin{equation}\label{sl.0}
\sl{n\!+\!1} \CC_0 := \{a\in \sl{n\!+\!1} \CC \;|\; \forall j, \; \sum_{i=1}^na_{ij} = 0=\sum_{i=1}^na_{ji}\}.
\end{equation}
\end{proposition}
\begin{proof}
In view of  Proposition~\ref{prop.Lie} (and  \eqref{Lie.ab} in particular) and the definition \eqref{san.Lie} of the Lie bracket on $ \sna n \CC$ one immediately finds that
\begin{equation}\label{comm}
[a,b]_o = (a-o)(b-o) -(b-o)(a-o) +o,
\end{equation}
where the right hand side uses the usual addition (subtraction) of matrices 
(this is not the same as addition in $\mathcal{V}(\sna n \CC_o)$, as the latter is relative to $o$ and hence defined by the formula $a-o+b$).  Matrices $a-o$ where $a\in \sna n\CC$ are traceless and the sum of elements in each row and column is 0, hence
$$
\{ a-o \;|\; a\in \sna n \CC\} = \sl{n+1} \CC_0 .
$$
The linear transformation 
$$
\mathcal{L}(\sna n \CC ;o)\to \sl{n+1} \CC_0, \qquad a\mapsto a-o,
$$
sends the Lie bracket $[-,-]_o$ to the commutator and hence it is the required isomorphism of Lie algebras. 
\end{proof}

\begin{proposition}\label{prop.san2}
For all $o\in \sna 2 \CC$, 
$$
\mathcal{L}(\sna 2 \CC ;o)\cong \sl 2\CC.
 $$
\end{proposition}
\begin{proof}
Since up to isomorphism the choice of $o$ makes no difference, let us choose $o:=A^{01}_{~0}$, set $\omega = e^{\frac 23 \pi i}$ and define
$$
e= \frac 1 3(A^{10}_{~0}+ \omega A^{00}_{~1}), \quad f= \frac 1 3(A^{10}_{~0}+ \omega^2 A^{00}_{~1}), \quad h = -\frac{\sqrt 3} 3 i A^{00}_{~0}.
$$
Then, using  relations \eqref{san2} one finds that 
$$
[h,e]_o= 2e, \quad [h,f]_o=-2f, \quad [e,f]_o =h,
$$
i.e., $h,e,f$ form the Chevalley basis for the Lie algebra $\sl 2 \CC$.
\end{proof}

\end{document}